\documentclass[11pt]{amsart}

\usepackage{amssymb,amsmath,epsfig,graphics}

\numberwithin{equation}{section}
\newcommand{\beq}{\begin{equation}}
\newcommand{\eeq}{\end{equation}}
\newcommand{\bea}{\begin{eqnarray}}
\newcommand{\eea}{\end{eqnarray}}
\newcommand{\beas}{\begin{eqnarray*}}
\newcommand{\eeas}{\end{eqnarray*}}

\newtheorem{theorem}{Theorem}[section]

\newtheorem{proposition}[theorem]{Proposition}
\newtheorem{corollary}[theorem]{Corollary}
\newtheorem{lemma}[theorem]{Lemma}
\newtheorem{remark}[theorem]{Remark}
\newtheorem{example}[theorem]{Example}
\newtheorem{examples}[theorem]{Examples}
\newtheorem{foo}[theorem]{Remarks}

\newcommand{\p}{\partial}

\newcommand{\bM}{\mathbb M}

\newcommand{\Rn}{\mathbb R^n}

\newcommand{\M}{\mathbb M}

\newcommand{\ve}{\varepsilon}

\title[Riesz transforms] {A note on the boundedness of Riesz transform for some subelliptic operators}

\author{Fabrice Baudoin}
\address{Department of Mathematics\\Purdue University \\
West Lafayette, IN 47907} \email[Fabrice Baudoin]{fbaudoin@math.purdue.edu}
\thanks{First author supported in part by
NSF Grant DMS-0907326}

\author{Nicola Garofalo}
\address{Department of Mathematics\\Purdue University \\
West Lafayette, IN 47907} \email[Nicola
Garofalo]{garofalo@math.purdue.edu}
\thanks{Second author supported in part by NSF Grant DMS-1001317}

\begin{document}

\maketitle

\begin{abstract}
Let $\M$ be a smooth connected non-compact manifold endowed with a smooth measure $\mu$ and a smooth locally subelliptic diffusion operator $L$ satisfying $L1=0$, and which is symmetric with respect to $\mu$. We show that if $L$ satisfies, with a non negative curvature parameter $\rho_1$,  the generalized curvature inequality in \eqref{CD} below,  then the  Riesz transform is bounded in $L^p (\bM)$ for every $p>1$, that is 
\[
\left\| \sqrt{\Gamma((-L)^{-1/2}f)} \right\|_p \le C_p \| f \|_p, \quad f \in C^\infty_0(\bM),
\]
where $\Gamma$ is the \textit{carr\'e du champ} associated to $L$. Our results apply in particular to all Sasakian manifolds whose horizontal Tanaka-Webster Ricci curvature is nonnegative, all Carnot groups with step two, and wide subclasses of principal bundles over Riemannian manifolds whose Ricci curvature is nonnegative.
\end{abstract}

\tableofcontents

\section{Introduction}

A central result in the analysis of $\Rn$ is the $L^p$ continuity of singular integrals in the range $1<p<\infty$. One basic consequence of this result is the $L^p$ boundedness of the Riesz transforms $\mathcal R_j = \frac{\p}{\p x_j} (-\Delta)^{-1/2}$, $j=1,...,n$, with their vector-valued counterpart 
\[
\mathcal R = (\mathcal R_1,...,\mathcal R_n) = \nabla (-\Delta)^{-1/2},
\]
see \cite{S}. In \cite{Str} Strichartz asked the question whether such $L^p$ continuity of the Riesz transform could be extended to non-compact Riemannian manifolds under suitable assumptions on the latter.  In this context the analogue of the vector-valued Riesz transform is the operator 
\[
\mathcal R  = \nabla \Delta^{-1/2}, 
\]
where we have denoted by $\Delta$ the Laplacian on $\M$ in its realization as a positive self-adjoint operator on $L^2(\M)$.
Strichartz's question is important for the purpose of developing analysis on manifolds.  To explain this point let us indicate by $(\cdot,\cdot)$ the inner product in $L^2(\M)$, and with $\Delta^{1/2}$ the positive self-adjoint square root of $\Delta$. Then one has the equality
\[
(\Delta f,f) = (\Delta^{1/2}f,\Delta^{1/2}f).
\]
This immediately gives 
\[
 \| | \nabla f | \|_2=\| \Delta^{1/2} f \|_2,
\]
which in turn allows to identify the first-order Sobolev subspaces of $L^2(\M)$ obtained by completion of $C^\infty_0(\M)$ with respect to the seminorms $\| | \nabla f | \|_2$ and $\| \Delta^{1/2} f \|_2$. Let us also notice in passing that the latter equality can be reformulated in terms of $\mathcal R$ as follows
\[
\| \mathcal R f\|_2 = \| f\|_2.
\]

However, when $1<p<\infty$ and $p \neq 2$, a similar identification of the two Sobolev spaces of order one obtained by completion of $C^\infty_0(\M)$ with respect to the seminorms $ \| | \nabla f | \|_p$ and $\| \Delta^{1/2} f \|_p$ is no longer such a simple matter. 
It is a well-known fact that an estimate such as 
\begin{align}\label{RZ}
A_p \| \Delta^{1/2} f \|_p \le \| | \nabla f | \|_p \le B_p \| \Delta^{1/2} f \|_p,\ \ \ \ \ f \in C_0^\infty(\bM),
\end{align}
would suffice for such identification. It is also known that the validity of the right-hand inequality in \eqref{RZ} for a certain $1<p<\infty$ implies that of  the left-hand inequality in $L^{p'}(\M)$, where $\frac 1p + \frac 1p' = 1$. 

Now the right-hand inequality in \eqref{RZ} is equivalent to the $L^p$ continuity of the Riesz operator $\mathcal R$.
It is then clear that \eqref{RZ} is true  for all $1<p<\infty$ if 
\begin{equation}\label{RZ2}
||\mathcal R f||_p \le C_p ||f||_p,\ \ \ \ \ f\in C^\infty_0(\M),
\end{equation}
for $1<p<\infty$, and this clarifies the relevance of the question raised by Strichartz.

\

An interesting result due to Bakry \cite{B} states that if the Ricci curvature of $\bM$ is bounded from below by a non negative constant then \eqref{RZ2}, and therefore \eqref{RZ} hold for every $1<p<\infty$. 
The purpose of the present note is to extend this result to a sub-Riemannian framework by using the generalized curvature-dimension inequality recently introduced by the authors in \cite{BG1}. 

\

This extension has been recently become possible thanks to a combination of the theory developed in the two papers \cite{BG1}, \cite{BBG}, with the remarkable results in \cite{CD}, \cite{ACDH}. The latter  two works have established that \eqref{RZ2} does hold in the range $1<p<\infty$ for complete, non-compact Riemannian manifolds satisfying suitable general assumptions which will be discussed below. In \cite{CD} the authors have proved that \eqref{RZ2} is true when $1<p\le 2$. In the paper \cite{ACDH} the authors have established \eqref{RZ2} in the remaining range $2\le p<\infty$. 
The essential new contribution of the present note is to verify that such general assumptions are verified (in a non-trivial manner) for a general class of locally subelliptic operators satisfying on a given smooth manifold $\M$ the generalized curvature-dimension inequality CD$(\rho_1,\rho_2,\kappa,d)$ in \eqref{CD} below, with curvature parameter $\rho_1\ge 0$ (this in the Riemannian case corresponds to Ric$ \ge 0$). Once this is done, the $L^p$ continuity of an appropriately defined Riesz operator will follow by the general real variable methods developed in \cite{CD}, \cite{ACDH}. 

\

To state the main result in this paper we assume that $\bM$ be a $C^\infty$ connected, non-compact manifold endowed with a smooth measure $\mu$. Throughout the paper, the notation $L^p(\M)$, $1\le p\le \infty$, indicates the space of $p$-summable functions on $\M$ with respect to the measure $\mu$. We assume that on $\M$ a second-order diffusion operator $L$ with real coefficients is given. We also suppose that $L$ be locally subelliptic, non-positive, and that it satisfy the assumptions listed in Section \ref{S:background}. There is a natural notion of (square of the length of the)``gradient'' associated with $L$, namely
\[
\Gamma(f) = \frac 12 \{L(f^2) - 2 f Lf\},
\]
 and a canonical distance $d$, see \eqref{di} below. We assume throughout that the metric space $(\M,d)$ be complete.
We also suppose that $\M$ be endowed with another bilinear differential form $\Gamma^Z$, see \eqref{gammaZ} below, and that $\Gamma$ and $\Gamma^Z$ satisfy all the hypothesis in Section \ref{S:background} below. From our perspective, the most significant assumption is the so-called generalized curvature-dimension inequality CD$(\rho_1,\rho_2,\kappa,d)$ in \eqref{CD} below, which we now recall for the reader's convenience: 

\

\emph{There exist constants $\rho_1 \ge 0$,  $\rho_2 >0$, $\kappa \ge 0$, and $d\ge 2$ such that the inequality 
\begin{equation}\label{CDi}
\Gamma_2(f) +\nu \Gamma_2^Z(f) \ge \frac{1}{d} (Lf)^2 +\left( \rho_1 -\frac{\kappa}{\nu} \right) \Gamma(f) +\rho_2 \Gamma^Z(f)
\end{equation}
 hold for every  $f\in C^\infty(\bM)$ and every $\nu>0$, where $\Gamma_2$ and $\Gamma_2^Z$ are defined by \eqref{gamma2} and \eqref{gamma2Z} below}.

\

The assumption \eqref{CDi} constitutes a sub-Riemannian generalization of the classical curvature-diemension inequality CD$(\rho,n)$ 
\[
\Gamma_2(f) \ge \frac 1n (\Delta f)^2 + \rho \Gamma(f),
\]
which, as a consequence of the well-known Bochner's identity, is known to hold on any $n$-dimensional Riemannian manifold satisfying Ric$ \ge \rho$.  

\

The parameter $\rho_1$ in CD$(\rho_1,\rho_2,\kappa,d)$ has the meaning of a lower bound on a sub-Riemannian Ricci tensor, see \cite{BG1} for extensive details. Throughout the present paper the assumption $\rho_1\ge 0$ will be in force. The semigroup $P_t = e^{t L}$ is a strongly continuous semigroup of contraction operators on $L^p(\M)$ for $1\le p\le \infty$. We denote by $p(x,y,t) = p(y,x,t)$ the positive heat kernel on $\M$ associated with the semigroup $P_t$. Given $f\in C^\infty_0(\M)$, the function 
\[
u(x,t) = P_t f(x) = \int_M f(y) p(x,y,t) d\mu(y),
\]
is a solution of the equation $Lu - u_t = 0$ in $\M\times (0,\infty)$, corresponding to the initial datum $u(x,0) = f(x)$, $x\in \M$. 

\

We now recall that, in the general framework described above, in the paper \cite{BG1} we proved a generalized Li-Yau type inequality for solutions of the heat equation on $\M$ of the form $u = P_t f$. From such inequality, we were able to derive several basic facts, among which the following off-diagonal Gaussian upper  bound: \emph{For any $0<\ve <1$
there exists a constant $C(\rho_2,\kappa,d,\ve)>0$, which tends
to $\infty$ as $\ve \to 0^+$, such that for every $x,y\in \bM$
and $t>0$ one has}
\begin{equation}\label{gue}
p(x,y,t)\le \frac{C(d,\kappa,\rho_2,\ve)}{V(x,\sqrt
t)^{\frac{1}{2}}V(y,\sqrt t)^{\frac{1}{2}}} \exp
\left(-\frac{d(x,y)^2}{(4+\ve)t}\right).
\end{equation}
Hereafter in this paper we adopt the notation
\[
V(x,r) = \mu(B(x,r)),
\]
where for $x\in \M$ and $r>0$ we have let $B(x,r) = \{y\in \M\mid d(y,x)<r\}$.

\

In the paper \cite{BBG} we further developed the program initiated in \cite{BG1} and were able to obtain the following basic result: \emph{There exists a constant $C_d>0$ depending only on $\rho_1, \rho_2, \kappa, d$, such that for every $x\in \M$ and $r>0$ one has}
\begin{equation}\label{dci} 
V(x,2r) \le C_d V(x,r).
\end{equation}
For a purely analytical proof of \eqref{dci} in the Riemannian setting we refer the reader to the paper \cite{BG2}.

Now in their work \cite{CD} the authors proved that the two results \eqref{gue} and \eqref{dci} are enough to establish the weak-$(1,1)$ continuity of the Riesz transforms for the space of homogeneous type $(\M,d)$. Since from integration by parts, and from the identity $(Lf,f) = \|(-L)^{1/2} f\|_2$  the strong $L^2$ continuity of the Riesz transform
\[
\|\sqrt{\Gamma((-L)^{1/2})f)}\|_2 = \|f\|_2
\]
trivially follows,
by the Marcinckiewicz interpolation theorem we thus obtain the following result.

\begin{theorem}\label{T:riesz12}
Let $1<p \le 2$.  There is a constant $C_p >0$ such that for every $f \in C^\infty_0(\M)$,
\begin{equation}\label{rtp}
\left\| \sqrt{\Gamma((-L)^{-1/2}f)} \right\|_{L^p(\M)} \le C_p \| f \|_{L^p(\M)}.
\end{equation}
\end{theorem}

Theorem \ref{T:riesz12} provides the $L^p$ continuity of the absolute value of the Riesz operator
\[
T f = \sqrt{\Gamma((-L)^{-1/2}f)},
\]
within the range $1<p\le 2$. We emphasize is that $T$ is a sublinear operator.
For the remaining range $2\le p<\infty$ we appeal to the real variable theory developed in the work \cite{ACDH}. We recall the salient ingredients of the general approach in that paper:
\begin{itemize}
\item[1)] $e^{tL} 1 = 1$ (stochastic completeness);
\item[2)]  global doubling condition;
\item[3)]  global Poincar\'e inequality;
\item[4)]  Caccioppoli type inequalities;
\item[5)] Gaffney type estimates;
\item[6)] bounds for $\sqrt t \sqrt{\Gamma(e^{tL})}$.
\end{itemize}
  
\

As for 1) the stochastic completeness in our framework follows as a special case of Theorem 3.5 in \cite{BG1}, see also \cite{Mu} for an extension of such result.
Regarding 2) we have already discussed \eqref{dci}. As for 3), we mention that in \cite{BBG} it was proved that there exists $C_p>0$, depending only on $\rho_1, \rho_2, \kappa, d$, such that for every $x\in \M$ and $r>0$ one has
\begin{equation}\label{pisr}
\int_{B(x,r)} |f - f_B|^2 d\mu \le C_p r^2 \int_{B(x,r)} \Gamma(f)  d\mu,
\end{equation}
for every $f\in C^1(\overline B(x,r))$. Thus 3. is available to us. 

\

We are thus missing ingredients 4), 5) and 6) In this note we establish these results, see Corollary \ref{C:caccioppoli}, Lemmas \ref{L:gaf1}, \ref{L:gaf2} and \ref{L:gaf3}, and Theorem \ref{T:ugb} below. This allows us to close the circle and, by using the work \cite{ACDH}, obtain the following result.

\begin{theorem}\label{T:riesz2infty}
Let $2\le p <\infty$.  There is a constant $C_p >0$ such that for every $f \in C^\infty_0(\M)$,
\begin{equation}\label{rtp}
\left\| \sqrt{\Gamma((-L)^{-1/2}f)} \right\|_{L^p(\M)} \le C_p \| f \|_{L^p(\M)}.
\end{equation}
\end{theorem}

By combining Theorems \ref{T:riesz12} and \ref{T:riesz2infty} we obtain the following result.

\begin{theorem}\label{T:equivalence}
Let $1<p<\infty$. There exist constants $A_p, B_p>0$ such that 
\begin{equation}\label{RZsr}
A_p \| (-L)^{1/2} f \|_p \le \| \sqrt{\Gamma(f)} \|_p \le B_p \| (-L)^{1/2} f \|_p,\ \ \ \ \ f \in C_0^\infty(\bM),
\end{equation}
\end{theorem}

\

The results in this paper establish the continuity of the Riesz transform and the equivalence of the Sobolev spaces defined by completion of $C^\infty_0(\M)$ with respect to the two seminorms in \eqref{RZsr} for the various classes of sub-Riemannian manifolds which are encompassed by the general framework of \cite{BG1}. While we refer the reader to that source for a detailed discussion of the examples, here we confine ourselves to mention the following basic result which is a corollary of our work.

\begin{theorem}\label{T:sasakian}
Let $(\bM,\theta)$ be a \emph{CR} manifold  with real dimension $2n+1$ and vanishing Tanaka-Webster torsion, i.e., a Sasakian manifold. If there exists $\rho_1\ge 0$ such that
for every $x\in \bM$ the Tanaka-Webster Ricci tensor satisfies the bound  
\[
\emph{Ric}_x(v,v)\ \ge \rho_1|v|^2,
\]
for every horizontal vector $v\in \mathcal H_x$,
then given any $1<p<\infty$ the Riesz transform associated with a sub-Laplacian on $\M$ is continuous on $L^p(\M)$.
\end{theorem}

In connection with Theorem \ref{T:sasakian} we mention that it was proved in \cite{BG1} that in the framework of Theorem \ref{T:sasakian} the generalized curvature-dimension inequality CD$(\rho_1,\rho_2,\kappa,d)$ does hold with $\rho_1 \ge 0$. Thus these manifolds fall within the scope of the assumptions in Section \ref{S:background}.

\

In closing, we mention some known partial results related to those in the present paper. In  \cite{LV} the boundedness of the Riesz transforms was proved on every stratified nilpotent Lie group. In \cite{A} this result was generalized to Lie groups of polynomial growth. 

\medskip

\noindent \textbf{Acknowledgment:} The second named author would like to thank Steve Hofmann for  several helpful discussions.
 
\bigskip

\section{Background}\label{S:background}

\subsection{Assumptions}\label{SS:frame}

Hereafter in this paper, $\bM$ will be a $C^\infty$ connected and non-compact manifold endowed with a smooth measure $\mu$. Throughout the paper, the notation $L^p(\M)$, $1\le p\le \infty$, indicates the space of $p$-summable functions on $\M$ with respect to the measure $\mu$. 

We assume that on $M$ a second-order diffusion operator $L$ with real coefficients is given. We also suppose that $L$ be locally subelliptic (for the relevant definition and properties of such operators see \cite{FSC} and \cite{JSC}), and that it satisfy:
\begin{itemize}
\item[1)] $L1=0$;
\item[2)] $\int_\bM f L g d\mu=\int_\bM g Lf d\mu$;
\item[3)] $\int_\bM f L f d\mu \le 0$,
\end{itemize}
for every $f , g \in C^ \infty_0(\bM)$. 

There is a natural gradient (or rather, a natural square of the length of a gradient) canonically associated with $L$, and it is given by the quadratic functional $\Gamma(f) = \Gamma(f,f)$, where
\begin{equation}\label{gamma}
\Gamma(f,g) =\frac{1}{2}(L(fg)-fLg-gLf), \quad f,g \in C^\infty(\bM).
\end{equation}
 The functional $\Gamma(f)$ is known as \textit{le carr\'e du champ}. Notice that $\Gamma(1) = 0$. Furthermore,
using the results in \cite{PS}, locally in the neighborhood of every point $x\in \M$ we can write
\begin{equation}\label{Lrep}
L =- \sum_{i=1}^m X_i^* X_i,
\end{equation}
where the vector fields $X_i$ are Lipschitz continuous (such representation is not unique, but this fact is of no consequence for us. We note for further reference that the number $m$ of vector fields entering in the local representation \eqref{Lrep} is bounded above by the dimension of the manifold $\M$). Therefore, for any $x\in \M$ there exists an open neighborhood $U_x$ such that in $U_x$ we have for any $f\in C^\infty(\M)$
\begin{equation}\label{Grep}
\Gamma(f)  = \sum_{i=1}^m (X_i f)^2.
\end{equation}
This shows that $\Gamma(f)\ge 0$ and it actually only involves differentiation of order one. Furthermore, the value of $\Gamma(f)(x)$ does not depend  on the particular representation \eqref{Lrep} of $L$.
With the operator $L$ we can also associate a canonical distance:
\begin{equation}\label{di}
d(x,y)=\sup \left\{ |f(x) -f(y) | \mid f \in  C^\infty(\bM) , \| \Gamma(f) \|_\infty \le 1 \right\},\ \ \  \ x,y \in \bM,
\end{equation}
where for a function $g$ on $\bM$ we have let $||g||_\infty = \underset{\bM}{\text{ess} \sup} |g|$. 

\

A tangent vector $v\in T_x\M$ is called \emph{subunit} for $L$ at $x$ if   
$v = \sum_{i=1}^m a_i X_i(x)$, with $\sum_{i=1}^m a_i^2 \le 1$, see \cite{FP}. It turns out that the notion of subunit vector for $L$ at $x$ does not depend on the local representation \eqref{Lrep} of $L$. A Lipschitz path $\gamma:[0,T]\to \M$ is called subunit for $L$ if $\gamma'(t)$ is subunit for $L$ at $\gamma(t)$ for a.e. $t\in [0,T]$. We then define the subunit length of $\gamma$ as $\ell_s(\gamma) = T$. Given $x, y\in \M$, we indicate with 
\[
S(x,y) =\{\gamma:[0,T]\to \M\mid \gamma\ \text{is subunit for}\ L, \gamma(0) = x,\ \gamma(T) = y\}.
\]
In this paper we assume that 
\[
S(x,y) \not= \varnothing,\ \ \ \ \text{for every}\ x, y\in \M.
\]
Under such assumption  it is easy to verify that
\begin{equation}\label{ds}
d_s(x,y) = \inf\{\ell_s(\gamma)\mid \gamma\in S(x,y)\},
\end{equation}
defines a true distance on $\M$. Furthermore, thanks to Lemma 5.43 in \cite{CKS} we know that
\[
d(x,y) = d_s(x,y),\ \ \ x, y\in \mathbb M,
\]
hence we can work indifferently with either one of the distances $d$
or $d_s$.

\

In addition to the differential form \eqref{gamma}, we assume that $\M$ be endowed with another smooth bilinear differential form, indicated with $\Gamma^Z$, satisfying for $f,g \in C^\infty(\M)$
\begin{equation}\label{gammaZ}
\Gamma^Z(fg,h) = f\Gamma^Z(g,h) + g \Gamma^Z(f,h),
\end{equation}
and $\Gamma^Z(f) = \Gamma^Z(f,f) \ge 0$. 
Given the first-order bilinear forms $\Gamma$ and $\Gamma^Z$ on $\bM$, we now introduce the following second-order differential forms:
\begin{equation}\label{gamma2}
\Gamma_{2}(f,g) = \frac{1}{2}\big[L\Gamma(f,g) - \Gamma(f,
Lg)-\Gamma (g,Lf)\big],
\end{equation}
\begin{equation}\label{gamma2Z}
\Gamma^Z_{2}(f,g) = \frac{1}{2}\big[L\Gamma^Z (f,g) - \Gamma^Z(f,
Lg)-\Gamma^Z (g,Lf)\big].
\end{equation}
Observe that if $\Gamma^Z\equiv 0$, then $\Gamma^Z_2 \equiv 0$ as well. As for $\Gamma$ and $\Gamma^Z$, we will use the notations  $\Gamma_2(f) = \Gamma_2(f,f)$, $\Gamma_2^Z(f) = \Gamma^Z_2(f,f)$.

We make the following assumptions that will be in force throughout the paper:

\

\begin{itemize}
\item[(H.1)] There exists an increasing
sequence $h_k\in C^\infty_0(\bM)$   such that $h_k\nearrow 1$ on
$\bM$, and \[
||\Gamma (h_k)||_{\infty} +||\Gamma^Z (h_k)||_{\infty}  \to 0,\ \ \text{as} \ k\to \infty.
\]
\item[(H.2)]  
For any $f \in C^\infty(\bM)$ one has
\[
\Gamma(f, \Gamma^Z(f))=\Gamma^Z( f, \Gamma(f)).
\]
\item[(H.3)] The  \emph{generalized curvature-dimension inequality} \emph{CD}$(\rho_1,\rho_2,\kappa,d)$ be satisfied with 
$\rho_1 \ge 0$, that is: There exist constants $\rho_1 \ge 0$,  $\rho_2 >0$, $\kappa \ge 0$, and $d\ge 2$ such that the inequality 
\begin{equation}\label{CD}
\Gamma_2(f) +\nu \Gamma_2^Z(f) \ge \frac{1}{d} (Lf)^2 +\left( \rho_1 -\frac{\kappa}{\nu} \right) \Gamma(f) +\rho_2 \Gamma^Z(f)
\end{equation}
 hold for every  $f\in C^\infty(\bM)$ and every $\nu>0$, where $\Gamma_2$ and $\Gamma_2^Z$ are defined by \ref{gamma2} and \ref{gamma2Z}. 
 \end{itemize}
 \ 

For example,  the assumptions (H.1)-(H.3) are satisfied in all Carnot groups of step two, and in all complete Sasakian manifolds whose horizontal Tanaka-Webster Ricci curvature is non negative. For further examples, including a wide class of bundles over Riemannian manifolds we refer the reader to \cite{BG1}.
 
In this framework:
\begin{itemize}
\item $L$ is essentially self-adjoint on $C^ \infty_0(\bM)$, so that by using the spectral theorem for the Friedrichs extension of $L$ in the Hilbert space $L^2 (\mathbb{M})$, we may construct a strongly continuous contraction semigroup $(P_t)_{t \ge 0}$ in $L^2 (\mathbb{M})$ whose infinitesimal generator is $L$;
\item By  hypoellipticity of $L$, $(P_t)_{t \ge 0}$ admits a heat kernel, that is: There is a smooth function $p(t,x,y)$, $t \in (0,\infty), x,y \in \mathbb{M}$, such that for every $f \in L^2 (\mathbb{M})$ and $x \in \mathbb{M}$ ,
\[
P_t f (x)=\int_{\mathbb{M}} p(t,x,y) f(y) d\mu (y).
\]
 Moreover, the heat kernel satisfies the two following conditions:
\begin{itemize}
\item[(i)] (Symmetry) $p(t,x,y)=p(t,y,x)$;
\item[(ii)] (Chapman-Kolmogorov relation) $p(t+s,x,y)=\int_{\mathbb{M}} p(t,x,z)p(s,z,y)d\mu(z)$; 
\end{itemize}
\item The semigroup  $(P_t)_{t \ge 0}$ is a sub-Markovian semigroup: If $ 0\le f \le 1$ is a function in $L^2 (\mathbb{M})$, then $0 \le P_t f \le 1$;
\item By the  Riesz-Thorin interpolation theorem, $(P_t)_{t \ge 0}$ defines a contraction semigroup on $L^p (\mathbb{M})$, $1 \le p \le \infty$.
\end{itemize} 

\subsection{Some known results}

In this section we collect some results that, under the above listed assumptions, were proved in the works \cite{BG1} and \cite{BBG}. Such results constitute the backbone of the present paper.

The first basic result is that the manifold $\M$ is \emph{stochastically complete} with respect to $L$, i.e., for every $t>0$
\begin{equation}\label{sc}
P_t 1 = e^{tL} 1 = 1.
\end{equation}
We recall that this result is equivalent to the uniqueness of the bounded solution of the Cauchy problem
\[
\begin{cases}
Lu - u_t = 0,\ \ \ \ \ \M\times (0,\infty),
\\
u(x,0) = \varphi(x),\ \ \ \ x\in \M,
\end{cases}
\]
with bounded initial datum $\varphi$. The property \eqref{sc} is a corollary of Theorem 3.5 in \cite{BG1}, see also \cite{Mu} for an extension of such result.

\

Another basic result is the following Gaussian upper bound that was proved in \cite{BG1}.

\begin{theorem}\label{T:ugb}
For any $0<\ve <1$
there exists a constant $C(\ve) = C(d,\kappa,\rho_2,\ve)>0$, which tends
to $\infty$ as $\ve \to 0^+$, such that for every $x,y\in \bM$
and $t>0$ one has
\[
 p(x,y,t)\le \frac{C(\ve)}{V(x,\sqrt
t)} \exp
\left(-\frac{d(x,y)^2}{(4+\ve)t}\right).
\]
\end{theorem}

In the paper \cite{BBG} it has been proved that the metric measure space $(\bM,d,\mu)$ satisfies the global volume doubling property and that, furthermore, the $L^2$  Poincar\'e inequality is satisfied on balls. More precisely, we have the following result. 

\begin{theorem}\label{T:main}
There exist constants $C_d, C_p>0$, depending only on $\rho_1, \rho_2, \kappa, d$, for which one has for every $x\in \M$ and every $r>0$:
\begin{equation}\label{dcsr}
V(x,2r) \le C_d\ V(x,r);
\end{equation}
\begin{equation}\label{pisr}
\int_{B(x,r)} |f - f_B|^2 d\mu \le C_p r^2 \int_{B(x,r)} \Gamma(f)  d\mu,
\end{equation}
for every $f\in C^1(\overline B(x,r))$.
\end{theorem}

We list for future use the following well-known consequence of the doubling condition \eqref{dcsr}.

\begin{corollary}\label{C:dc}
With $C_d$ as in \eqref{dcsr} define
\[
Q = \log_2 C_d.
\]
Then, for every $x\in \M$ and any $0<r<R<\infty$ one has
\begin{equation}\label{dcallscales}
V(x,R) \le C_d \left(\frac Rr\right)^Q V(x,r).
\end{equation}
In particular, if $y, z\in \M$ and $t>0$ we have
\begin{equation}\label{dct}
V(y,\sqrt{ t}) \leq C_d \left(\frac {d(y,z)}{\sqrt t} + 1\right)^Q V(z,\sqrt{ t}), 
\end{equation}
and also for any given $\alpha>0$, there exists a constant $C>0$ depending on $C_d$ and $\alpha$, such that 
\begin{equation}\label{exp}
\int_\M \exp\left(-\alpha \frac{d(y,z)^2}{t}\right) d\mu(z) \le C V(y,\sqrt t).
\end{equation}
\end{corollary}

\begin{proof}
The proof of \eqref{dcallscales} is standard and we omit it. As for \eqref{dct} it is enough to apply \eqref{dcallscales} with $R = d(y,z) + \sqrt t$ and $r = \sqrt t$, to obtain
\[
V(y,\sqrt{ t}) \le V(z,d(y,z)+\sqrt{ t})
                 \leq C_d  \left(\frac {d(y,z)+\sqrt{ t}}{\sqrt t} \right)^Q V(z,\sqrt{ t}) .
\]
Finally, \eqref{exp} easily follows by covering $\M$ with dyadic rings $B(y,2^{k+1}\sqrt t)\setminus B(y,2^k \sqrt t)$, and then using \eqref{dcsr}.
\end{proof}

\section{Riesz transform}\label{S:RT}

Our objective in this section is proving that the Riesz  transform associated to $L$ is bounded in $L^p(\bM)$. \footnote{In the case where $\mu(\bM) <\infty$ one has to consider the space $L_0^p(\bM)$  of functions in $L^p(\bM)$ with mean 0, see \cite{ACDH}; this modification will be implicit in the text.}

\begin{theorem}\label{T:riesz}
Let for every $1<p<\infty$. There is a constant $C_p >0$ such that for every $f \in L^p(\M)$,
\begin{equation}\label{rtp}
\left\| \sqrt{\Gamma((-L)^{-1/2}f)} \right\|_{L^p(\M)} \le C_p \| f \|_{L^p(\M)}.
\end{equation}
\end{theorem}

As a consequence of Theorem \ref{T:riesz}, we obtain the following result.

\begin{theorem}\label{T:equivalence2}
Let $1<p<\infty$. There exist constants $A_p, B_p>0$ such that 
\begin{equation}\label{RZsr}
A_p \| (-L)^{1/2} f \|_p \le \| \sqrt{\Gamma(f)} \|_p \le B_p \| (-L)^{1/2} f \|_p,\ \ \ \ \ f \in C_0^\infty(\bM).
\end{equation}
\end{theorem}

\subsection{The case $1<p \le 2$}

Following our discussion in the introduction, the boundedness of the Riesz transform on $L^p(\M)$ for $1<p\le 2$ follows by combining Theorem \ref{T:ugb} and \eqref{dcsr} in Theorem \ref{T:main} above with Theorem 1.1 in \cite{CD}. We note explicitly that \eqref{dcsr} implies that $(M,d,\mu)$ is a space of homogeneous type according to \cite{CW}, see also \cite{C}, and so all tools of real analysis are available. From these results the weak-$(1,1)$ continuity of the Riesz transform
\[
\mu(\{x\in \M\mid \sqrt{\Gamma((-L)^{-1/2}f)(x))} >\lambda\}) \le \frac{C}{\lambda} \|f\|_{L^1(\M)},\ \ \ \lambda>0,
\]
can be established as in \cite{CD}. Then, for the range $1<p\le 2$ the inequality \eqref{rtp} follows by applying Marcinckiewicz real interpolation theorem. The reader should notice that the operator
\[
T = \sqrt{\Gamma((-L)^{-1/2}f))}
\]
is a sublinear operator, i.e., $|T(f+g)(x)|\le |Tf(x)| + |Tg(x)|$ for every $f, g$ and a.e. $x\in \M$.

\subsection{The case $p>2$}

Following the general method in the proof of Theorem 3.1 in \cite{ACDH}, to establish the boundedness of the Riesz transform when $p >2$, we need to the ingredients listed as 1)-6) in the introduction. As it was mentioned there the items which are at this point missing are Caccioppoli  and Gaffney type estimates, as well as bounds for $\sqrt t \sqrt{\Gamma(e^{tL})}$. This section is devoted to filling this gap. We begin with establishing the former type of result.

\subsubsection{Caccioppoli type estimates}

In what follows we prove an a priori inequality of Caccioppoli type for the heat semigroup $P_t$. 

\begin{proposition}\label{P:caccioppoli}
Let $f \in C^\infty_0(\bM)$. For $t \ge  0$ we have
\[
  \Gamma(P_t f)  +\rho_2 t  \Gamma^Z(P_t f)\le \frac{1+\frac{2\kappa}{\rho_2} }{2t} (P_t ( f^2) -(P_tf)^2).
 \]
\end{proposition}

\begin{proof}

Let us fix $T>0$. Given a function $f\in C_0(\bM)$, for $0\le t\le T$ we introduce the  functionals

\[
\phi_1 (x,t)=\Gamma (P_{T-t}f)(x),
\]
\[
\phi_2 (x,t)= \Gamma^Z (P_{T-t}f)(x),
\]
which are defined on $\M\times [0,T]$. It is is easy to check that, with $\Gamma_2$ and $\Gamma_2^Z$ defined as in \eqref{gamma2} and \eqref{gamma2Z}, we have
\[
L\phi_1+\frac{\partial \phi_1}{\partial t} =2  \Gamma_2 ( P_{T-t}f). 
\]
and
\[
L\phi_2+\frac{\partial \phi_2}{\partial t} =2  \Gamma_2^Z ( P_{T-t}f),
\]
see also \cite{BG1}. 
Consider now the function
\begin{align*}
\phi (x,t)&= a(t) \phi_1 (x,t)+b(t) \phi_2(x,t) \\
 & =a(t)\Gamma ( P_{T-t}f)(x)+b(t) \Gamma^Z ( P_{T-t}f)(x),
\end{align*}
where $a$ and $b$ are two nonnegative functions  that will be chosen later.
At this point we observe that, since by hypothesis $\rho_1\ge 0$, the generalized curvature-dimension CD$(\rho_1,\rho_2,\kappa,d)$ in \eqref{CD} trivially implies the generalized curvature-dimension inequality CD$(0,\rho_2,\kappa,\infty)$. Applying the latter with the choice $\nu = \frac ba$ we thus obtain
\begin{align*}
  L\phi+\frac{\partial \phi}{\partial t} &=
a' \Gamma (P_{T-t}f)+b' (P_{T-t} f) \Gamma^Z ( P_{T-t}f)  +2a  \Gamma_2 (P_{T-t}f)+2b (P_{T-t} f) \Gamma_2^Z ( P_{T-t}f) \\
&\ge  \left(a' -2\kappa \frac{a^2}{b}\right) \Gamma (P_{T-t}f)  +(b'+2\rho_2 a) \Gamma^Z ( P_{T-t}f).
\end{align*}
Let us now chose
\[
a(t)=\frac{1}{\rho_2}(T-t),\ \ \ \ \ b(t)=(T-t)^2.
\]
With this choice we have on $[0,T]$,
\[
b'+2\rho_2 a\equiv 0,
\]
and
\[
a' -2\kappa \frac{a^2}{b}=-\frac{1}{\rho_2}-\frac{2 \kappa}{\rho_2^2}. 
\]
We find then
\[
L\phi+\frac{\partial \phi}{\partial t} \ge \left( -\frac{1}{\rho_2}-\frac{2 \kappa}{\rho_2^2} \right) \Gamma (P_{T-t}f).
\]
From a comparison theorem for parabolic partial differential equations (see for instance p.52 in \cite{Fried} or Proposition 3.2 in \cite{BG1})  we deduce
\begin{equation}\label{cac}
P_T(\phi(\cdot,T))(x) \ge \phi(x,0)  -\left(\frac{1}{\rho_2}+\frac{2 \kappa}{\rho_2^2} \right) \int_0^T P_t (\Gamma (P_{T-t}f) ) dt.
\end{equation}
To conclude, we consider the functional
\[
\Psi(t) = \frac 12 P_t\left((P_{T-t} f)^2)\right).
\]
A straightforward computation shows that
\begin{align*}
\Psi'(t) & = \frac 12 P_t\left(L(P_{T-t} f)^2)\right) +  P_t\left(P_{T-t} f \frac{\p}{\p t}(P_{T-t} f)\right)
\\
& = \frac 12 P_t \bigg(L(P_{T-t} f)^2) -  2 P_{T-t} f L(P_{T-t} f)\bigg) =  P_t\left(\Gamma(P_{T-t} f)\right).
\end{align*}
This gives
\[
 \int_0^T P_t (\Gamma (P_{T-t}f) dt= \Psi(T) - \Psi(0) = \frac{1}{2} \left( P_T (f^2) -(P_Tf)^2 \right).
\]
Replacing this information in \eqref{cac}, along with the identities
\[
\phi(x,0)=\frac{1}{\rho_2} T \Gamma (P_{T} f) + T^2 \Gamma^Z (P_{T} f), \quad \ \ \phi(x,T)=0,
\]
we reach the desired conclusion.

\end{proof}

As a corollary of Proposition \ref{P:caccioppoli} and of the $L^p$ continuity of $P_t$, for $1\le p\le \infty$, we obtain the following Caccioppoli type estimates.

\begin{corollary}\label{C:pcaccioppoli}
For  any $f\in C^\infty_0(\M)$, $2 \le p \le \infty$ we have
\[
 \| \sqrt{ \Gamma(P_t f) } \|_p  \le \sqrt{  \frac{1+\frac{2\kappa}{\rho_2} }{2t} } \| f \|_p.
 \]
\end{corollary}

In what follows we will also need the following result.

\begin{corollary}\label{C:caccioppoli}
Let $n$ be the dimension of $\M$. With $C = \sqrt{\frac{(2\kappa+\rho_2)n}{2\rho_2}}$ we have
\begin{equation}\label{L1}
\sup_{x \in \bM} \int_\bM \sqrt{\Gamma ( p( \cdot,y,t) )(x)  } d\mu(y) \le  \frac{C}{\sqrt{t}}. 
\end{equation}
This estimate implies that for any $f \in L^\infty(\bM)$,
\begin{equation}\label{Linfty}
 \| \sqrt{ \Gamma(P_t f) } \|_\infty  \le \frac{C}{\sqrt{t}} \| f \|_\infty.
 \end{equation}
\end{corollary}

\begin{proof}
 For $x\in \M$ fixed, let then $U_x$ be a sufficiently small neighborhood of $x$ in which we can write $L$ as in \eqref{Lrep}, 
where the vector fields $X_i$ are Lipschitz continuous. For  $f \in C^\infty_0(\bM)$, we have
\[
\left|  \int_\bM X_i p( \cdot,y,t) (x) f(y) d\mu (y) \right| = |X_i P_t f (x)|  \le \| \sqrt{ \Gamma(P_t f) } \|_\infty \le \sqrt{ \frac{1+\frac{2\kappa}{\rho_2} }{2t} } \| f \|_\infty
\]
 Let $\varepsilon >0$ and take now
 \[
 f(y)= h_k(y) \frac{X_i p( \cdot,y,t) (x) }{|X_i p( \cdot,y,t) (x)|+\varepsilon} ,
 \]
 where $0 \le h_k \le 1$ is an increasing  sequence in $C^\infty_0(\bM)$, converging to $1$. We obtain
 \[
 \int_\bM  h_k (y) \frac{X_i p( \cdot,y,t) (x)^2 }{|X_i p( \cdot,y,t) (x)|+\varepsilon}  d\mu(y) \le \sqrt{ \frac{1+\frac{2\kappa}{\rho_2} }{2t} }
 \]
 By the monotone convergence theorem and by Fatou's theorem we deduce, first letting $k \to \infty$ and then $\varepsilon \to 0$, that
 \[
 \int_\bM |X_i p( \cdot,y,t) (x)| d\mu(y) \le  \sqrt{ \frac{1+\frac{2\kappa}{\rho_2} }{2t}}.
 \]
The estimate \eqref{L1} follows immediately from the latter inequality. 
 
\end{proof}

\subsubsection{Gaffney-type estimates}\label{SS:gaffney}

We now turn to the second main ingredient which are Gaffney type estimates. In what follows we indicate with $E, F\subset \M$ two closed subsets. We need the following results.

\begin{lemma}\label{L:gaf1}
For every two closed sets $E, F\subset \M$, and any $f\in L^\infty(\M)$ supported in $E$, one has
\[
||P_t f||_{L^2(F)} \le e^{-\frac{d(E,F)^2}{4t}} ||f||_{L^2(E)}.
\]
\end{lemma} 

\begin{proof}
Let us suppose $d(E,F)>0$, otherwise the conclusion follows trivially from the $L^2$ continuity of $P_t$. As a first step we let $\psi$ denote a function in Lip$(\M)$, such that $\Gamma(\psi)\le 1$ a.e. on $\M$. With $\alpha>0$ and $\phi = e^{-\alpha \psi}$, we note that
\begin{equation}\label{gammaphi}
\Gamma(\phi)  = \alpha^2 \phi^2 \Gamma(\psi) \le  \alpha^2 \phi^2.
\end{equation}
Let $f\in L^\infty(\M)$, and consider 
\[
y(t) = ||\phi P_t f||^2_{L^2(\M)}.
\]
Denoting with $<\cdot,\cdot>$ the inner product in $L^2(\M)$, we have
\begin{align*}
y'(t) & = 2 <L(P_t f),\phi^2 P_t f> = - 2 \int_\M \Gamma(P_t f,\phi^2 P_t f) d\mu
\\
& = - 2 \int_\M \phi^2 \Gamma(P_t f) d\mu - 4  \int_\M \phi P_t f \Gamma(P_t f,\phi) d\mu
\\
& \le - 2 \int_\M \phi^2 \Gamma(P_t f) d\mu + \frac 2\ve   \int_\M \phi^2  \Gamma(P_t f) d\mu +  2 \ve \alpha^2  \int_\M \phi^2  (P_t f)^2 d\mu, 
\end{align*}
where to estimate the last term we have used \eqref{gammaphi}. Choosing $\ve = 1$ we conclude
\[
y'(t) \le 2\alpha^2 y(t),
\]
and, upon integrating this inequality, we obtain
\begin{equation}\label{phipt}
||\phi P_t f||_{L^2(\M)} \le e^{\alpha^2 t} ||\phi f||_{L^2(\M)}.
\end{equation}

We now want to show that \eqref{phipt} yields the desired conclusion. Suppose that supp$ f \subset E$, then we argue as follows. We take $\psi(x) = d(x,F)$, and for any $\alpha >0$ we let $\phi = e^{- \alpha \psi}$. Since $\phi \equiv 1$ on $F$, from \eqref{phipt} we obtain
\[
||P_t f||_{L^2(F)} \le ||\phi P_t f||_{L^2(\M)} \le e^{\alpha^2 t} ||\phi f||_{L^2(\M)} =  e^{\alpha^2 t} ||\phi f||_{L^2(E)}.
\]
Now, for any $x\in E$ we have $\psi(x) \ge d(E,F)$, and thus $\phi \le e^{-\alpha d(E,F)}$ on $E$. This gives
\[
||P_t f||_{L^2(F)} \le  e^{\alpha^2 t - \alpha d(E,F)} ||f||_{L^2(E)}.
\]
By choosing $\alpha = \frac{d(E,F)}{2 t}>0$ we reach the desired conclusion.  

\end{proof}

For $\omega>0$ sufficiently small denote 
\[
S_\omega = \{z = r e^{i\theta}\in \mathbb C\mid 0<r<\infty,\ |\theta|<\frac \pi2 + \omega\}.
\]
Since $L$ generates an analytic semigroup $e^{zL}$ in a sector $S_\omega$, we have 
\[
P_t = \frac{1}{2\pi i} \int_{\Gamma_\delta}  e^{t\zeta} R(\zeta;L) d\zeta,
\]
where $\Gamma_\delta\subset S_\omega$ is the path composed of the two rays $r e^{i\theta}$ and $r e^{-i\theta}$, with $0<r<\infty$ and $\frac \pi2 <\theta<\frac \pi2 +\delta$, $0<\delta<\omega$, and $R(\zeta;L)$ is the resolvent of $L$. Using the same argument as in Lemma \ref{L:gaf1} it is easy to see that the Gaffney estimate \eqref{phipt} continues to be valid for $P_z f$ with $z\in S_\delta$ for any fixed $0<\delta<\omega$. Similarly to what was done above  this leads to an estimate of the type
\begin{equation}\label{phipt2}
||P_z f||_{L^2(F)} \le  e^{\alpha^2 |\Re z| - \alpha d(E,F)} ||f||_{L^2(E)},\ \ \ \ z\in S_\delta,
\end{equation}
for any $f\in L^\infty(\M)$ which is supported in $E$, and any $\alpha>0$.  

\begin{lemma}\label{L:gaf2}
There exists a constant $C>0$ such that for every two closed sets $E, F\subset \M$, and any $f\in L^\infty(\M)$ supported in $E$, one has
\[
t ||L P_t f||_{L^2(F)} \le C e^{-\frac{d(E,F)^2}{6t}} ||f||_{L^2(E)}.
\]
\end{lemma} 

\begin{proof}
Consider the path 
\[
\Gamma_t = \{z = t\left(1+\frac{e^{i\theta}}{2}\right)\in \mathbb C\mid 0\le \theta \le 2\pi\}\subset S_\omega.
\]
Using the analyticity of $P_t$ in the sector $S_\omega$ we can write
\[
t \frac{\p P_t}{\p t} = \frac{t}{2\pi i} \int_{\Gamma_t} \frac{e^{z L}}{(z - t)^2} dz = \frac 1\pi \int_0^{2\pi} e^{-i\theta} e^{t\left(1+\frac{e^{i\theta}}{2}\right)L} d\theta.
\]
Using \eqref{phipt2}, this gives
\begin{align*}
t ||LP_t f||_{L^2(F)} & \le \frac{1}{\pi} \int_0^{2\pi} ||P_{t\left(1+\frac{e^{i\theta}}{2}\right)} f||_{L^2(F)}
\\
& \le 2 e^{\alpha^2 t(1+\frac{|\cos \theta|}{2}) - \alpha d(E,F)} ||f||_{L^2(E)}
\\
& \le 2 e^{\frac{3}{2} \alpha^2 t - \alpha d(E,F)} ||f||_{L^2(E)}
\end{align*}
Choosing $\alpha = \frac{d(E,F)}{3t}$ yields the desired conclusion.

\end{proof}

Since the distance function $y\to d(x,y)$ is obviously a Lipschitz continuous function on $\M$ (with respect to $d$ itself), in view of \eqref{Lrep} and \eqref{Grep}, and of a Rademacher type theorem for Lipschitz vector fields we can construct Lipschitz continuous cut-off functions on metric balls, see for instance Theorem 1.5 in \cite{GN}. We collect this fact in the following lemma.

\begin{lemma}\label{L:bumps}
Let $0<s<t<\infty$. There exists a constant $C>0$ such that for every $B(x,s)\subset B(x,t)\subset \M$ there exists a function $\varphi\in$ \emph{Lip}$(\M)$, with $0\le \varphi\le 1$, $\varphi \equiv 1$ on $B(x,s)$, and supp$\ \varphi \subset B(x,t)$, for which
\[
\sqrt{\Gamma(\varphi)}\le \frac{C}{t-s}.
\]
\end{lemma}

\begin{corollary}\label{C:bumpsforsets}
Given a closed set $F\subset \M$, consider the open set $F_\ve = \{x\in \M\mid d(x,F)<\ve\}$. There exists a function $\varphi\in $ \emph{Lip}$(\M)$ such that $0\le \varphi\le 1$, $\varphi \equiv 1$ on $F_{\ve/2}$, and $\varphi \equiv 0$ in $\M\setminus F_\ve$, and for which
\[
\sqrt{\Gamma(\varphi)} \le \frac{C}{\ve}.
\]
\end{corollary}

\begin{proof}
It follows from Lemma \ref{L:bumps} by a standard partition of unity argument.

\end{proof}

We are now in a position to establish the third Gaffney type estimate which we will need. 

\begin{lemma}\label{L:gaf3}
There exist constants $C \ge 0$ and $\alpha >0$ such that for any two closed sets $E, F\subset \M$,  and every function $f \in L^\infty(\M)$ supported in $E$, one has
\[
\sqrt{t} \| \sqrt{\Gamma (P_t f)} \|_{L^2(F)} \le C e^{ -\alpha \frac{d(E,F)^2}{t}} \| f\|_{L^2(E)}.
\]
\end{lemma}

\begin{proof}
We adapt the argument on p. 930 in \cite{ACDH}. If $d(E,F)\le\sqrt t$ there is nothing to prove. We can thus assume that $d(E,F)> \sqrt t$. With $\ve = \frac{d(E,F)}{3}$, consider the set $F_\ve$, and pick a function $\varphi\in $ Lip$(\M)$ supported in $F_\ve$ as in Corollary \ref{C:bumpsforsets}. We have
\begin{align*}
 t \int_\M \varphi^2 \Gamma(P_t f) d\mu & = \frac t2 \int_\M \varphi^2 \left[L\left((P_t f)^2\right) - 2 P_t f L(P_t f)\right] d\mu
\\
& = - t \int_\M \varphi^2 P_t f LP_t f d\mu - 2t \int_\M \varphi P_t f \Gamma(\varphi,P_t f) d\mu
\\
& \le t ||L P_t f||_{L^2(F_\ve)} ||P_t f||_{L^2(\M)}  + 2 \left(t \int_\M\varphi^2 \Gamma(P_t f) d\mu\right)^{1/2}\left(t \int_\M \Gamma(\varphi) (P_t f)^2 d\mu\right)^{1/2}.
\end{align*}
We now use Lemma \ref{L:gaf2} and the $L^2$ continuity of $P_t$ to obtain for some $\alpha>0$
\[
 t ||L P_t f||_{L^2(F_\ve)} ||P_t f||_{L^2(\M)} \le C e^{-\frac{d(E,F_\ve)^2}{6t}} ||f||^2_{L^2(E)}\le C e^{-\alpha \frac{d(E,F)^2}{t}} ||f||^2_{L^2(E)}. 
\]
Recalling that $\sqrt t < d(E,F)$, and using the support property of $\Gamma(\varphi)$, and the estimate $\Gamma(\varphi) \le C d(E,F)^{-2}$ from Corollary \ref{C:bumpsforsets}, we conclude
\begin{align*}
\left(\sqrt t \int_\M \Gamma(\varphi) (P_t f)^2 d\mu\right)^{1/2} & \le C \left(\int_{F_\ve} (P_t f)^2 d\mu\right)^{1/2} 
\\
& \le C e^{-\frac{d(E,F_\ve)^2}{4t}} ||f||_{L^2(E)} \le C e^{-\alpha \frac{d(E,F)^2}{t}} ||f||_{L^2(E)},
\end{align*}
for an appropriate $\alpha>0$.
We conclude
\[
t \int_\M \varphi^2 \Gamma(P_t f) d\mu  \le C e^{-\alpha \frac{d(E,F)^2}{t}} ||f||^2_{L^2(E)} + C e^{-\alpha \frac{d(E,F)^2}{t}}  ||f||_{L^2(E)} \left(t \int_\M \varphi^2 \Gamma(P_t f) d\mu\right)^{1/2}.
\]
A trivial estimate allows to conclude that
\[
t \int_\M \varphi^2 \Gamma(P_t f) d\mu  \le C e^{-\alpha \frac{d(E,F)^2}{t}} ||f||^2_{L^2(E)}.
\]
Since
\[
t \| \sqrt{\Gamma (P_t f)} \|^2_{L^2(F)} \le t \int_\M \varphi^2 \Gamma(P_t f) d\mu,
\]
we have reached the desired conclusion.

\end{proof}

\subsubsection{The completion of the proof Theorem \ref{T:riesz} in the range $2\le p<\infty$}

We are now in a position to prove Theorem \ref{T:riesz} in the range $2\le p<\infty$. As mentioned above, all we need to do at this point is combine Theorem  \ref{T:main}, Corollary \ref{C:caccioppoli} and  Lemmas \ref{L:gaf1}, \ref{L:gaf2} and \ref{L:gaf3} with the real variable results in \cite{ACDH}.

\section{Pointwise gradient estimates of the heat kernel}\label{S:pge}

In this section we investigate the validity of pointwise gradient estimates of the heat kernel. Besides being interesting in their own right, such estimates are also connected with Theorem \ref{T:riesz} in the range $2\le p<\infty$. For a detailed discussion of this aspect we refer the reader to \cite{ACDH}.

\begin{theorem}\label{T:ugb}
There exists a constant $C = C(d,\kappa,\rho_2)>0$ such that for every $x,y\in \M$
and $t>0$ one has
\[
\sqrt{\Gamma\left(p(\cdot,y,t)\right)}(x)  \le \frac{C}{\sqrt t V(y,\sqrt t)}.
\]
\end{theorem}

\begin{proof}
We fix a point $x\in M$ and $s>0$ and begin with the observation that if we consider the function 
\[
f_{x,s}(z) = p(x,z,s),
\]
then thanks to Theorem \ref{T:ugb} there exists a constant $C>0$ (independent of $x\in \M$ and $s>0$) such that $f_{x,s}\in L^\infty(\M)$ and 
\[
\left\|f_{x,s}\right\|_{L^\infty(\M)} \le \frac{C}{V(x,\sqrt s)}.
\]
We next observe that, given points $x, y\in \M$ and $t>0$, then we can write
\[
p(x,y,t) = P_{\frac t2}(f_{y,\frac t2})(x).
\]
For $x\in \M$ fixed, let then $U_x$ be a sufficiently small neighborhood of $x$ in which we can write $L$ as in \eqref{Lrep}. We thus have
\begin{align*}
|X_i p(\cdot,y,t)(x)| & = |X_i\left(P_{\frac t2}(f_{y,\frac t2})\right)(x)|\le \sqrt{\Gamma\left(P_{\frac t2}(f_{y,\frac t2})\right)}(x)
\\
& \le \left\|\sqrt{\Gamma\left(P_{\frac t2}(f_{y,\frac t2})\right)}\right\|_{L^\infty(\M)} \le \frac{C}{\sqrt t} \left\|f_{y,\frac t2}\right\|_{L^\infty(\M)}
\\
& \le \frac{C}{\sqrt t V(y,\sqrt{t/2})},
\end{align*}
where in the last inequality we have used the $L^\infty$ Caccioppoli inequality in Corollary \ref{C:caccioppoli}. 
This estimate finally gives
\[
\sqrt{\Gamma\left(p(\cdot,y,t)\right)}(x) = \sqrt{\sum_{i=1}^d \left(X_i p(\cdot,y,t)(x)\right)^2} \le \frac{C}{\sqrt t V(y,\sqrt{t/2})}.
\]
The proof is completed by an application of \eqref{dcsr} in Theorem \ref{T:main}.

\end{proof}

\begin{theorem}\label{T:gei}
For every $\ve>0$ there exists a constant $C(\ve) = C(d,\kappa,\rho_2,\ve)>0$ such that for every $x,y\in \M$
and $t>0$ one has
\[
\sqrt{\Gamma\left(p(\cdot,y,t)\right)}(x)  \le \frac{C(\ve)}{\sqrt t V(y,\sqrt t)} \exp\left(-\frac{d(x,y)^2}{4(1+\ve)t}\right).
\]
\end{theorem}

\begin{proof}
It suffices to appeal to Theorem 4.11 in \cite{CS}. In such result the authors, by a beautiful use of Phragm\'en-Lindel\"of theory, prove that the combination of \eqref{dcallscales} in Corollary \ref{C:dc}, and of the estimates
\[
p(x,x,t)\le \frac{C}{V(x,\sqrt t)},
\]
and
\[
\sqrt{\Gamma\left(p(\cdot,y,t)\right)}(x)  \le \frac{C}{\sqrt t V(y,\sqrt t)},
\]
allows to improve the estimate in Theorem \ref{T:ugb} into the following one
\[
\sqrt{\Gamma\left(p(\cdot,y,t)\right)}(x)  \le \frac{C}{\sqrt t V(y,\sqrt t)}\left(1 + \frac{d(x,y)^2}{4t}\right)^{1+3Q} \exp\left(-\frac{d(x,y)^2}{4t}\right),
\]
where $Q$ is as \eqref{dcallscales}. From the latter estimate the desired conclusion follows.

\end{proof}


\vskip 0.3in

\begin{thebibliography}{99}


\bibitem[ACDH]{ACDH}
P. Auscher, T. Coulhon,  X.T. Duong, S. Hofmann, 
\emph{Riesz transform on manifolds and heat kernel regularity} 
Ann. Sci. \'Ecole Norm. Sup. (4) \textbf{37}~(2004), no. 6, 911-957. 


\bibitem[A]{A} G. Alexopoulos, 
\emph{An application of homogenization theory to harmonic analysis: Harnack inequalities and Riesz transforms on Lie groups of polynomial growth.} 
Canad. J. Math. 44 (1992), no. 4, 691--727. 

\bibitem[B]{B} D. Bakry,
\emph{\'Etude des transformations de Riesz dans les vari\'et\'es riemanniennes \`a courbure de Ricci minor\'ee}. (French) [A study of Riesz transforms in Riemannian manifolds with minorized Ricci curvature] S\'eminaire de Probabilit\'es, XXI, 137--172,
Lecture Notes in Math., 1247, Springer, Berlin, 1987. 



\bibitem[BG1]{BG1} 
F. Baudoin \& N. Garofalo, \emph{Curvature-dimension inequalities and Ricci lower bounds for sub-Riemannian manifolds with transverse symmetries}, Arxiv preprint 1101.3590, submitted paper (2009).


\bibitem[BG2]{BG2}
\bysame, \emph{Perelman's entropy and doubling property on Riemannian manifolds}, J. Geom. Anal., to appear.

\bibitem[BBG]{BBG}
F. Baudoin. M. Bonnefont \& N. Garofalo, \emph{A sub-Riemannian curvature-dimension inequality, volume doubling property and the Poincar\'e inequality},  arXiv:1007.1600, submitted, (2010).

\bibitem[CKS]{CKS}
E. Carlen, S. Kusuoka \& D. Stroock, \emph{Upper bounds for
symmetric Markov transition functions},
Ann. Inst. H. Poincar\'e Probab. Statist.  \textbf{23}~ (1987),  no. 2, suppl.,
245--287.


\bibitem[C]{C}
M. Christ, \emph{A $T(b)$ theorem with remarks on analytic capacity and the Cauchy integral}, Colloq. Math. \textbf{60/61}~(1990), no. 2, 601Ð628. 

\bibitem[CW]{CW}
R. R.  Coifman \& G. Weiss,  \emph{Analyse harmonique non-commutative sur certains espaces homog\`enes. (French) \'Etude de certaines int\'egrales singuli\`eres}, Lecture Notes in Mathematics, Vol. 242. Springer-Verlag, Berlin-New York, 1971. v+160 pp.

\bibitem[CD]{CD} T. Coulhon, X. T. Duong, \emph{Riesz transforms for $1 \le p \le 2$}, Trans.  Amer. Math. Soc., \textbf{351} \textbf{3}~(1999), 1151-1169.

\bibitem[CS]{CS} T. Coulhon, A. Sikora, 
\emph{Gaussian heat kernel upper bounds via the PhragmŽn-Lindelšf theorem}. 
Proc. Lond. Math. Soc. (3) 96 (2008), no. 2, 507--544. 



\bibitem[FP]{FP}
C. Fefferman \& D. H. Phong, \emph{Subelliptic eigenvalue problems},
Conference on harmonic analysis in honor of Antoni Zygmund, Vol. I,
II (Chicago, Ill., 1981),  590--606, Wadsworth Math. Ser.,
Wadsworth, Belmont, CA, 1983.

\bibitem[FSC]{FSC}
C. L. Fefferman \& A. S\'anchez-Calle, \emph{Fundamental solutions
for second order subelliptic operators},  Ann. of Math. (2)
\textbf{124}~(1986), no. 2, 247--272.




\bibitem[F]{Fried}
A. Friedman, \emph{Partial differential equations of parabolic type}, Dover, 2008.


\bibitem[GN]{GN}
N. Garofalo \& D.-M. Nhieu, \emph{Lipschitz continuity, global smooth approximations and extension theorems for Sobolev functions in Carnot-CarathŽodory spaces}, J. Anal. Math. \textbf{74}~(1998), 67-97.


\bibitem[JSC]{JSC}
D. Jerison \& A. S\'anchez-Calle, \emph{Subelliptic second order differential operators}, Lecture. Notes in Math., 1277 (1987), pp. 46-77.




\bibitem[LV] {LV} Lohou\'e, N., Varopoulos, N.: \emph{Remarques sur les transform\'ees de Riesz sur les
groupes nilpotents}. C.R.A.S. Paris, 301, 11 (1985) 559-560.

\bibitem[Mu]{Mu}
I. Munive, \emph{Generalized curvature-dimension inequalities, stochastic completeness and volume growth}, preprint, 2011.

\bibitem[PS]{PS}
R. S. Phillips \& L. Sarason, \emph{Elliptic-parabolic equations of the second order}, J. Math. Mech. \textbf{17}~1967/1968, 891-917.

\bibitem[S]{S}
E. M. Stein, \emph{Singular integrals and differentiability properties of functions}, Princeton Univ. Press, 1970.

\bibitem[Str]{Str}
R. Strichartz, \emph{Analysis  of the Laplacian on the complete Riemannian manifold}, J. Funct. Anal. \textbf{52}~(1983), 48-79.  




\end{thebibliography}
\end{document}